\documentclass
[11pt,a4paper,english]
{article} 
\usepackage{import}
\usepackage{blindtext}

\usepackage[utf8]{inputenc}
\usepackage{babel}                 
\usepackage{float}   
\usepackage{url}     
\usepackage{sectsty}               
\usepackage{tabularx}              
\usepackage{titling}               
\usepackage{imakeidx}              
\usepackage{xcolor}                
\usepackage{enumitem}              
\usepackage{amsmath}                                    
\usepackage{amssymb}                                    
\usepackage{amsthm}
\usepackage[Gray,squaren,thinqspace,thinspace]{SIunits} 
\usepackage{listings}                                   
\usepackage[small,bf,hang]{caption}        
\usepackage{subcaption}                    
\usepackage{sidecap}                       
\usepackage[colorlinks=true,citecolor=red,linkcolor=blue]{hyperref} 
\usepackage[capitalise,noabbrev]{cleveref}
\usepackage{url}                           
\usepackage{cite}                          
\usepackage{tikz}          
\usepackage{pgfplots}      
\usepackage{pgfplotstable} 
\usepackage{placeins}      
\usepackage{graphicx}      
\usepackage{longtable}     
\usepackage{mathrsfs}
\usepackage{braket}
\usepackage{bbm}
\pgfplotsset{compat=1.18} 
\usepackage{tikz-cd}
\usetikzlibrary{positioning}
\usetikzlibrary{decorations.pathmorphing}
\usetikzlibrary{knots}
\usetikzlibrary{calc}
\usepackage[export]{adjustbox}

\theoremstyle{plain}
\newtheorem{thm}{Theorem}[section]
\newtheorem{prop}[thm]{Proposition}
\newtheorem*{prop*}{Proposition}
\newtheorem{lem}[thm]{Lemma}

\newtheorem*{thm*}{Theorem}

\theoremstyle{definition} 
\newtheorem{ex}[thm]{Example}
\newtheorem{defn}[thm]{Definition}

 \newtheorem{notation}[thm]{Notation}

\Crefname{thm}{Theorem}{Theorems}
\Crefname{prop}{Proposition}{Propositions}
\Crefname{lem}{Lemma}{Lemmas}
\Crefname{cor}{Corollary}{Corollaries}
\Crefname{defn}{Definition}{Definitions}
\Crefname{tab}{Table}{Tables}
\Crefname{ex}{Example}{Examples}
\Crefname{chap}{Champter}{Chapters}
\Crefname{app}{Appendix}{Appendices}

\usepackage{stmaryrd}

\newcommand{\cp}{\mathrm{CP}}

\newcommand{\cpinf}{\cp^\infty}
\newcommand{\Rep}[1]{\mathrm{Rep}\br{#1}}

\newcommand{\hpair}[2]{\langle #1 , #2 \rangle}
\newcommand{\ra}{\rightarrow}

\newcommand{\End}[1]{\mathrm{End}(#1)}
\newcommand{\br}[1]{(#1)}                               

\newcommand{\ot}{\otimes}

\newcommand{\id}{\mathrm{id}}

\newcommand{\CC}{\mathbb{C}}

\newcommand{\WCat}{W^*\text{-}\mathrm{Cat}}
\newcommand{\WAlg}{W^*\text{-}\mathrm{Alg}}

\newcommand{\Mod}{\mathrm{Mod}}

\newcommand{\Endo}{\mathrm{End}}
\newcommand{\Hom}{\mathrm{Hom}}
\newcommand{\ignore}[1]{}
\newcommand{\op}{\mathrm{op}}

\newcommand{\Chan}{\mathrm{Chan}}

\newcommand{\Hilb}{\mathrm{Hilb}}

\makeatletter
\usepackage{authblk}

\usepackage{geometry}
\begin{document}

\title{$\cpinf$ and beyond: 2-categorical dilation theory
}

\author[1,2]{Robert Allen\thanks{r.allen@bristol.ac.uk}}
\author[1]{Dominic Verdon\thanks{dominic.verdon@bristol.ac.uk}}
\affil[1]{School of Mathematics, University of Bristol}
\affil[2]{Heilbronn Institute for Mathematical Research}

\maketitle

\begin{abstract}
The problem of extending the insights and techniques of categorical quantum mechanics to infinite-dimensional systems was considered in (Coecke and Heunen, 2016). In that work the $\mathrm{CP}^{\infty}$-construction, which recovers the category of Hilbert spaces and quantum operations from the category of Hilbert spaces and bounded linear maps, was defined. Here we show that by a `horizontal categorification' of the $\mathrm{CP}^{\infty}$-construction, one can recover the category of all von Neumann algebras and channels (normal unital completely positive maps) from the 2-category of von Neumann algebras, bimodules and intertwiners. As an application, we extend Choi's characterisation of extremal channels between finite-dimensional matrix algebras to a characterisation of extremal channels between arbitrary von Neumann algebras. 
\end{abstract}
\section{Introduction}

\subsection{Motivation}

\paragraph{Representations of quantum channels.}

One approach to quantum theory is to identify physical systems with von Neumann algebras of observables on those systems, while dynamical maps between systems are identified with normal (i.e. weak $*$-continuous)  unital completely positive (CP) maps between these observable algebras, which are conventionally called \emph{channels}. A key advantage of this algebraic approach to quantum theory is that it includes not only quantum-to-quantum, but also classical-to-quantum, quantum-to-classical and classical-to-classical dynamics~\cite[\S{}3.2]{Holevo2003}.

The Hilbert space formulation of quantum theory is recovered from the von Neumann algebraic formulation by a representation result for quantum channels called \emph{Stinespring's theorem}, which we now state in an appealing form given in~\cite[Cor. 12]{CH16}. We write $B(X)$ for the von Neumann algebra of bounded operators on a Hilbert space $X$.
\begin{quote}
Let $X,Y$ be Hilbert spaces and let $f: B(X) \to B(Y)$ be a normal CP map. There exists a pair $(E,V)$ of a Hilbert space $E$ (the \emph{environment}) and a bounded linear map $V: Y \to X \otimes E$, such that the channel is defined by
$$
f(a) = V^{\dagger} \circ (a \otimes \id_E) \circ V
$$
We call the pair $(E,V)$ a \emph{representation} of the completely positive map $f$. The map $f$ is unital iff $V$ is an isometry. 
\end{quote}
\noindent
As a string diagram in the category $\Hilb$ of all Hilbert spaces and bounded linear maps (we read diagrams from bottom to top):
\begin{align}\label{eq:hilbstinspringconj}
    \includegraphics[valign=c]{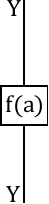}~~=~~\includegraphics[valign=c]{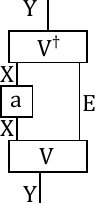}
\end{align}

\noindent
The limitation of this formulation of Stinespring's theorem is that it only applies to maps between von Neumann algebras of all bounded operators on a Hilbert space. 

\paragraph{Categorical quantum mechanics and CP$^\infty$.} 
The programme of categorical quantum mechanics~\cite{Abramsky2004,Heunen2019} studied quantum theory from the perspective of the $W^*$-tensor category $\Hilb$ of Hilbert spaces and linear maps. 
One goal of categorical quantum mechanics was to recover von Neumann algebras and channels via a categorical construction from $\Hilb$. 

In the case of finite-dimensional (f.d.) von Neumann algebras, this is possible using Frobenius algebras. A special symmetric Frobenius algebra in $\Hilb$ corresponds precisely to a f.d. von Neumann algebra equipped with its canonical special trace~\cite{Vicary2011}. Furthermore, CP maps between f.d. von Neumann algebras can be identified with morphisms between the corresponding Frobenius algebras obeying a certain positivity condition~\cite{Coecke2016}\cite[\S{}7.2.1]{Heunen2019}. The algebraic theory can therefore be recovered as a theory of Frobenius algebras in $\Hilb$, in the finite-dimensional case. 

Unfortunately, this does not generalise to infinite-dimensional von Neumann algebras, since Frobenius algebras, being self-dual, are necessarily finite-dimensional. To resolve this problem an alternative construction, the \emph{$\mathrm{CP}^{\infty}$ construction}, was proposed~\cite{CH16}. The objects of the category $\mathrm{CP}^{\infty}(\Hilb)$ are objects of $\Hilb$, while morphisms $X \to Y$ are equivalence classes of isometries $Y \to X \otimes E$ in $\Hilb$, where two isometries are equivalent iff they induce the same map $B(X) \to B(Y)$ by the conjugation~\eqref{eq:hilbstinspringconj}. By the formulation of Stinespring's theorem stated above, this construction recovers the category of von Neumann algebras $B(H)$ of all bounded operators on a Hilbert space and channels between them. However, it does not recover the category of all von Neumann algebras and channels.

 \paragraph{\texorpdfstring{$W^*$-tensor categories}{W*-tensor categories.}}
A dagger category with direct sums is a $W^*$-category if and only if every endomorphism space is a von Neumann algebra. This can be thought of as a horizontal categorification of a von Neumann algebra.
$W^*$-categories were first defined in~\cite{GLR85}, and developed further in \cite{Yamagami2007,FW19,CHJP22}, among others.
$W^*$-categories have found applications in compact quantum groups, subfactors and algebraic quantum field theory.

The tensor category $\Hilb$ is the prototypical example of a $W^*$-category. In this work we will make essential use of this $W^*$-structure to recover the theory of  von Neumann algebras and channels from $\Hilb$. In fact, we will find that this calls for a further categorification, from the $W^*$-tensor category $\Hilb$ to the $W^*$-2-category of $W^*$-categories, normal functors and bounded natural transformations.



\subsection{Our results}

In this work we generalise the Stinespring theorem stated above to normal completely positive maps between arbitrary von Neumann algebras. To achieve this we allow the CP map to be represented not just in the $W^*$-tensor category $\Hilb$, but in the $W^*$-2-category of von Neumann algebras, bimodules and intertwiners~\cite{Landsman2000}. We are thereby able to construct the category of all von Neumann algebras and channels by a natural generalisation of the CP$^{\infty}$-construction. 

\subsubsection{Two equivalent \texorpdfstring{$W^*$-2-categories}{W*-2-categories}}\label{sec:introw*}

We begin by defining the 2-category $\WAlg$, introduced in \cite[\S{}5]{Landsman2000} under the name $[W^*]$:

\begin{itemize}
    \item \emph{Objects}: von Neumann algebras $A,B,\dots$.
    \item \emph{1-morphisms:} $X,Y,\dots: A \to B$ are $A,B$-bimodules, with composition given by Connes' fusion tensor product (which we write as $- \otimes -$).
    \item \emph{2-morphisms:} $f,g,\dots:X \to Y$ are bounded linear maps intertwining the bimodule actions.
\end{itemize}
Explicit definitions of these structures will be recalled in Section~\ref{sec:bimodules}. The 2-category $\WAlg$ is a $W^*$-2-category~\cite{Yamagami2007}; in particular, its $\Hom$-categories are $W^*$-categories~\cite{GLR85}. 

There is an alternative, functorial definition (due to~\cite{Yamagami2007}) of an equivalent $W^*$-2-category, which we will denote $\WCat$. Given some basic assumptions on $W^*$-categories which will be stated in Section~\ref{sec:wcats}, this functorial definition is as follows:
\begin{itemize}
    \item \emph{Objects:} $W^*$-categories. 
    \item \emph{1-morphisms:} Normal unitary linear functors. 
    \item \emph{2-morphisms:} Bounded natural transformations. 
\end{itemize}
\noindent
The equivalence is the usual one, which takes a von Neumann algebra to its $W^*$-category of right modules; in particular, the object $\mathbb{C}$ of $\WAlg$ (i.e. the 1-dimensional von Neumann algebra) is taken to the object $\Hilb$ in $\WCat$.

\begin{notation}
Our results are indifferent to which of the equivalent 2-categories $\WAlg$ and $\WCat$ are being considered, so we will use notation which is also independent of this choice. We will use $r,s,t,\dots$ for objects, independently of whether they are $W^*$-categories or von Neumann algebras; $X,Y,Z,\dots$ for 1-morphisms; and $f,g,h,\dots$ for 2-morphisms. The objects $\mathbb{C}$ in $\WAlg$, and $\Hilb$ in $\WCat$ will play a special role in what follows, and we will use the name $r_0$ for either of these objects. 
\end{notation}


\subsubsection{A classification of von Neumann algebras in \texorpdfstring{$\WAlg$}{W*-Alg}}

Let $r$ be any object of $\WAlg$, and let $X$ be a 1-morphism $r_0 \to r$. The algebra $\Endo(X)$ is a von Neumann algebra.  (This is simply a consequence of $\WAlg$ being a $W^*$-2-category.) 

Let $\Hom(r_0,r)$ be the $W^*$-category of $1$-morphisms $r_0 \to r$. We say that a 1-morphism $X: r_0 \to r$ is a \emph{generator} in this category, or a \emph{generating 1-morphism}, if every other 1-morphism $r_0 \to r$ is a subobject of a (possibly infinite) direct sum of copies of $X$; that is, there exists an isometric 2-morphism from the object into the direct sum. Every von Neumann algebra is obtained as $\Endo(X)$ for some generating 1-morphism $X$ out of $r_0$ in $\WAlg$. This extends to a classification of von Neumann algebras in terms of such 1-morphisms. We say that two 1-morphisms $X: r_0 \to r$ and $Y: r_0 \to s$ are \emph{unitarily equivalent} if there exists an equivalence 1-morphism $E: r \to s$ and a unitary 2-morphism $U: Y \to X \otimes E$. 
\begin{prop}\label{prop:vnclassification}
Let $X: r_0 \to r$ and $Y: r_0 \to s$ be generating 1-morphisms in $\WAlg$. The von Neumann algebras $\Endo(X)$ and $\Endo(Y)$ are:
\begin{itemize}
    \item Morita equivalent iff the objects $r$ and $s$ are equivalent in $\WAlg$.
    \item $*$-isomorphic iff there is a unitary equivalence $X \simeq Y$.
\end{itemize}
\end{prop}
\noindent
Here Morita equivalence for von Neumann algebras is the standard notion found in~\cite{Rie74}.

\subsubsection{Stinespring's theorem}

We are ready to state the 2-categorical generalisation of Stinespring's theorem. We say that a 2-morphism $f: X \to Y$ in a $W^*$-2-category is an \emph{isometry} if $f^{\dagger} \circ f = \id_X$, and a \emph{partial isometry} if $f^{\dagger} \circ f$ is a projection in the $W^*$-algebra $\End{X}$.
\begin{thm}[Generalised Stinespring's theorem]\label{thm:stinespring}
Let $r,s$ be objects of $\WAlg$, and let $X: r_0 \to r$ and $Y: r_0 \to s$ be generating 1-morphisms. 

Let $f: \Endo(X) \to \Endo(Y)$ be a normal completely positive map. Then there exists a 1-morphism $E: r \to s$ (the `environment') and a 2-morphism $V: Y \to X \otimes E$ in $\WAlg$, such that 
$$
f(a) = V^{\dagger} \circ (a \otimes \id_{E}) \circ V
$$
We call the pair $(E,V)$ a \emph{representation} of the CP map $f$.  The CP map $f$ is unital iff $V$ is an isometry.

Two pairs $(E,V)$ and $(E',V')$ are representations of the same CP map iff there exists a partial-isometric 2-morphism $\sigma: E \to E'$ such that 
\begin{align}\label{eq:fintertwiner}
V' = (\id \otimes \sigma) \circ V && V = (\id \otimes \sigma^{\dagger}) \circ V'
\end{align}
Indeed, every CP map $f: \Endo(X) \to \Endo(Y)$ has a \emph{minimal representation}: that is, an initial object in the dagger category $\Rep{f}$ whose objects are Stinespring representations $(E,V: Y \to X \otimes E)$ of $f$ and whose morphisms $(E,V) \to (E',V')$ are 2-morphisms $\sigma: E \to E'$ in $\WAlg$ satisfying~\eqref{eq:fintertwiner}.
\end{thm}
\noindent
As a string diagram in the $W^*$-2-category $\WAlg$:
\begin{align*}
\includegraphics[valign=c]{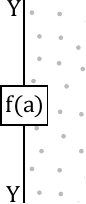} ~~~~~=~~~~~ \includegraphics[valign=c]{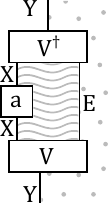}.
\end{align*}
Here we have left the regions corresponding to the object $r_0$ unshaded, and shaded the regions corresponding to the object $r$ with wavy lines and the regions corresponding to the object $s$ with polka dots.

Note that all the usual properties of a minimal representation follow from it being an initial object in a dagger category; it is unique up to a unitary satisfying~\eqref{eq:fintertwiner} and is related to every other representation by a unique isometry satisfying~\eqref{eq:fintertwiner}.

We think it worthwhile to state the following characterisation of $*$-homomorphisms, which follows straightforwardly from the categorical framework. 
\begin{prop}\label{prop:*hom}
    Let $f: \Endo(X) \to \Endo(Y)$ be a normal CP map and let $(E,V: X \to Y \otimes E)$ be a minimal representation. The map $f$ is a unital $*$-homomorphism iff the 2-morphism $V$ is unitary. 
\end{prop}

\subsubsection{\texorpdfstring{The CP$^\infty$-construction}{The CP-infinity construction}}

The 2-categorical generalisation of the CP$^{\infty}$ construction is now straightforward to state. Let $\mathrm{CP}$ be the category whose objects are von Neumann algebras and whose 1-morphisms are normal completely positive maps.  Let $\Chan \subset \mathrm{CP}$ be the category whose objects are von Neumann algebras and whose morphisms are channels, i.e. unital normal completely positive maps. 

The category $\mathrm{CP}$ can then be constructed from $\WAlg$ as follows:
    \begin{itemize}
        \item Objects: Generating 1-morphisms out of the object $r_0$ in $\WAlg$.
        \item Morphisms: Let $X: r_0 \to r$, $Y: r_0 \to s$ be objects of $\mathrm{CP}$. Then a morphism $X \to Y$ is an equivalence class of 2-morphisms of type $Y \to X \otimes E$ in $\WAlg$, where $E: r \to s$ is any 1-morphism. 

        The equivalence relation on these 2-morphisms is defined as follows: we say that two 2-morphisms $V: Y \to X \otimes E$ and $V': Y \to X \otimes E'$ are equivalent iff there exists a partial isometry $\sigma: E \to E'$ satisfying~\eqref{eq:fintertwiner}.
    \end{itemize}
The 2-category $\Chan$ can likewise be constructed from $\WAlg$ by restricting to equivalence classes of \emph{isometric} 2-morphisms of type $Y \to X \otimes E$. 

Note that we recover the original $\mathrm{CP}^{\infty}$-construction if we restrict objects of $\mathrm{CP}$ to generating 1-morphisms of type $r_0 \to r_0$ (since every object of the category $\Hilb$ is a generator). Our construction can therefore be thought of as a sort of horizontal categorification of the original $\mathrm{CP}^{\infty}$-construction. 

\subsection{An application to extremal channels}

The 2-categorical framework we have introduced has allowed us to construct the category of all von Neumann algebras and channels via the generalised $\mathrm{CP}^{\infty}$-construction. To show that 2-categorical dilation theory is useful more generally, we use it to extend Choi's characterisation of extremal points in the convex set of channels between finite-dimensional matrix algebras~\cite[Thm. 5]{Cho75} to extremal points in the convex set of channels between arbitrary von Neumann algebras. 
(C.f.~\cite[Thm. 32]{WW17}, which characterises extremal channels in terms of an injective affine order isomorphism, and~\cite[Prop 4.1]{Mohari2018}, which holds for channels whose target is of the form $B(H)$.)
\begin{prop}\label{prop:choi}
    Let $X:r_0 \to r$ and $Y: r_0 \to s$ be generating 1-morphisms in $\WAlg$. Let $f: \Endo(X) \to \Endo(Y)$ be a channel with minimal representation $(E,V: X \to Y \otimes E)$. Then $f$ is extremal iff, for any $m \in \Endo(E)$,
\begin{align}\label{eq:extremalityintro}
\includegraphics[valign=c]{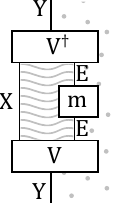}~~=0 ~~~~~\Rightarrow~~~~~m=0
\end{align}
\end{prop}
\begin{ex}[Correspondence between pure states and irreducible modules]
To illustrate the proposition, we will rederive the well-known correspondence between pure states and irreducible modules over von Neumann algebras. We will use the algebraic definition of $\WAlg$. Let $r$ be some von Neumann algebra, let $X: r_0 \to r$ be any generating right $r$-module and let $\mathbb{C}: r_0 \to r_0$ be the 1-dimensional Hilbert space. A normal state of $\Endo(X)$ is a channel $f: \Endo(X) \to \Endo(\mathbb{C})$, which by Theorem~\ref{thm:stinespring} has a minimal representation of the form 
$$
f(a)~~=~~\includegraphics[valign=c]{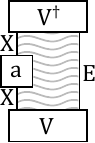}
$$
Here $E$ is a left module over $r$ and $V: \mathbb{C} \to X \otimes E$ is an isometry picking out a unit-norm vector in the Hilbert space $X \otimes E$. We claim that the state is pure (i.e. extremal) iff $E$ is an irreducible $r$-module. For `only if', suppose that $E= E_1 \oplus E_2$ is a direct sum of $r$-modules, and let $\pi_1, \pi_2 \in \Endo(E)$ be the corresponding orthogonal projections satisfying $\pi_1 + \pi_2 = \id_E$. Then define:
\begin{align*}
V_1~~:=~~\includegraphics[valign=c]{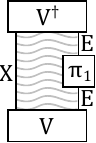} && V_2~~ :=~~\includegraphics[valign=c]{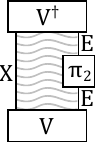}
\end{align*}
By minimality of the representation, neither $V_1$ nor $V_2$ are zero, since $V_1 = 0$ implies by positivity that $(\id_X \otimes \pi_2) \circ V = V$, implying that $V$ is not an initial object in $\Rep{f}$, and likewise for $V_2 = 0$. One may therefore set $m:= V_2 \pi_1 - V_1 \pi_2$ to violate~\eqref{eq:extremalityintro}. For `if', suppose that $E$ is irreducible; then $m$ must be a scalar multiple of $\id_E$. But since $V$ is an isometry, we obtain~\eqref{eq:extremalityintro}.
\end{ex}

\subsection{Further applications}

\paragraph{Supermaps} 

Transformations of channels are called superchannels or supermaps.
Any quantum supermap can be realised as a quantum circuit with a hole where the input channel is placed \cite{CDP08}. This powerful result has lead to applications in many areas where channels are transformed or compared, such as cloning, discrimination, estimation, tomography and programming. In an upcoming paper, we utilise the same graphical calculus as this paper, but with the 2-category of finite-dimensional 2-Hilbert spaces. We prove that any supermap acting on channels between von Neumann algebras satisfies a similar realisation theorem, where the quantum channels are replaced by channels in this generalised sense. This unlocks the potential to generalise already powerful results to a broader range of transformations, by the incorporation of classical as well as quantum information. 

\paragraph{Covariant quantum mechanics}

The results from this paper could generalise to the case where there is a group symmetry. For example, by considering the $C^*$-2-category of $G$-equivariant Hilbert $C^*$-bimodules, for some compact quantum group $G$, one should be able to define internal Homs using the decomposition over irreducible representations on $G$, and hence state a generalised covariant Stinespring's theorem.

\subsection{Related work}

\paragraph{Finite dimension.} We have already mentioned that CP maps between finite-dimensional von Neumann algebras can be identified with morphisms between Frobenius algebras in $\Hilb$ obeying a certain positivity condition \\ ~\cite{Coecke2016}\cite[\S{}7.2.1]{Heunen2019}. In fact, this is a consequence of the more fundamental Theorem~\ref{thm:stinespring}; since in finite dimensions all 1-morphisms are dualisable, the diagrams can be deformed to obtain a map between Frobenius algebras.

\paragraph{Previous Stinespring theorems.} Stinespring's theorem is usually stated for channels of type $A \to B(H)$, where $A$ is an arbitrary von Neumann algebra~\cite{Stinespring1955}. The idea of using bimodules over von Neumann algebras (or, in the $C^*$-algebraic setting, Hilbert $C^*$-correspondences) in order to generalise to an arbitrary target algebra is by no means new (see e.g.~\cite[\S{}5]{Paschke1973}\cite{Kasparov1980}\cite{Pop1986}\cite{Sza09}\cite{WW17}).
In \cite{PY10}, Pellonpää proves both a Stinespring's theorem and a classification of extremal maps. Rather than arbitrary algebras, the targets of these maps are spaces of $A$-sesquilinear maps from $V \times V \ra A$, for some $C^*$-algebra $A$ and $A$-module $V$.

\paragraph{Paschke dilations.} We particularly remark on~\cite{WW17}, which is closely related to this work. In that paper the authors define Paschke dilations: for a channel $f: A \to B$ between von Neumann algebras, a Paschke dilation is a tuple $(P,\rho,\nu)$, where $P$ is a von Neumann algebra, $\rho: A \to P$ is a normal unital $*$-homomorphism, and $\nu: P \to B$ is a normal CP map, such that $f = \nu \circ \rho$ and the tuple $(P,\rho,\nu)$ satisfies a universal property with respect to other such decompositions. The relation between Paschke dilations and the Stinespring representations discussed here is as follows. Let $f: \Endo(X) \to \Endo(Y)$ be a normal CP map, and let $(E,V: Y \to E)$ be a minimal Stinespring representation; then the tuple $(\Endo(X \otimes E), \rho, \mathrm{ad}_V)$ is a Paschke dilation, where $\rho: \Endo(X) \to \Endo(X \otimes E)$ is the embedding $x \mapsto x \otimes \id_E$ and $\mathrm{ad}_V$ is the CP map $x \to V^{\dagger} \circ x \circ V$.

\paragraph{Other approaches to infinite-dimensional categorical quantum mechanics.}

Aside from the $\mathrm{CP}^{\infty}$-construction, several other approaches have been suggested to generalise categorical quantum mechanics to infinite-dimensional systems. Since Frobenius algebras in $\Hilb$ are unital if and only if they are finite-dimensional, non-unital Frobenius algebras have been studied~\cite{AH12}. More radically, one can use non-standard analysis to legitimise working with infinite sums and define a unit for Frobenius algebras in a category of non-standard separable Hilbert spaces~\cite{GG17}. An approach using dagger linearly distributive categories has also been suggested~\cite{CCS21}.

    

\subsection{Acknowledgements}

RA was supported by the Additional Funding Programme for Mathematical Sciences, delivered by EPSRC (EP/V521917/1) and the Heilbronn Institute for Mathematical Research. DV was supported by the European Research Council (ERC) under the European Union’s Horizon 2020 research and innovation programme (grant agreement No.817581).

\section{Background}

\subsection{\texorpdfstring{$W^*$-2-categories}{W*-2-categories}}\label{sec:wcats}

Throughout we will assume that the reader has a grounding in category theory, including 2-category theory, at the level of~\cite{Heunen2019}. We will not assume a strong background in operator algebra. We do, however, assume that the reader is familiar with the definition of a von Neumann algebra.

\begin{defn}[{\cite[Defs. 1.1 and 2.1]{GLR85}}]
    A $\dagger$-category is a category equipped with an involutive contravariant endofunctor which acts as the identity on objects.
    A $W^*$-category is a $\CC$-linear $\dagger$-category with a conjugate linear dagger operation, enriched in the category of Banach spaces, such that
    \begin{itemize}
        \item every morphism $f$ satisfies $||{f^\dag \circ f}|| = || f ||^2$,
        \item every pair of composable morphisms satisfy $||f \circ g|| = || f || \ || g ||$,
        \item for every morphism $f\in \Hom(A,B)$, there exists a morphism $g \in \End{A}$ such that $f^\dag \circ f = g^\dag \circ g$,
        \item every Hom space admits a predual, as a Banach space.
    \end{itemize}
\end{defn}
We say that a linear functor between $W^*$-categories is \emph{normal} if it induces weak $*$-continuous maps on Hom-spaces, and \emph{unitary} if it preserves the dagger. We will here assume that our $W^*$-categories are essentially small, possess infinite direct sums and a zero object, and that all idempotents within them split. To this end, we recall the definition of an infinite direct sum:
\begin{defn}[{\cite[Before Prop. 6.5]{GLR85}}]\label{def:infdirsum}
Let $I$ be some index set. We say that an object $A$ of a $W^*$-category is a direct sum of a family of objects $\{A_i\}_{i \in I}$ if there exist isometries $w_i: A_i \to A$ such that $w_i^{\dagger} \circ w_j = \delta_{ij}$, and $\sum_{i \in I} w_i \circ w_i^{\dagger} = \id_A$ (where the sum is taken to converge in the weak $*$-topology).
\end{defn}
\noindent
Definition~\ref{def:infdirsum} can equivalently be phrased in terms of a universal property~\cite[Thm. 5.1]{Fritz2020}.

We will use the word `2-category' for the weak notion, also known as a bicategory. We will use the symbol $\otimes$ for horizontal composition and $\circ$ for vertical composition. (This will not agree with the usual notation for the Connes fusion tensor product, but as category theorists we prefer to reserve $\boxtimes$ for the monoidal product on a monoidal $W^*$-2-category. 
The 2-categories $\WAlg$ and $\WCat$ are indeed symmetric monoidal under spatial tensor product of von Neumann algebras and tensor product of Hilbert spaces~\cite{BDH14}, respectively, but we do not use that structure in this work.)

\begin{defn}
A $W^*$-2-category is a 2-category such that every $\Hom$-category is a $W^*$-category, the composition functors are linear and unitary, and the associator and unitor isomorphisms for composition are unitary.
\end{defn}
\begin{defn}
    We say that a 2-morphism $f: X \to Y$ in a dagger 2-category, or a morphism $f: X \to Y$ in a dagger category, is:
    \begin{itemize}
        \item An \emph{isometry} if $f^{\dagger} \circ f = \id_X$.
        \item A \emph{coisometry} if $f \circ f^{\dagger} = \id_Y$.
        \item \emph{Unitary} if it is both an isometry and a coisometry. 
        \item A \emph{partial isometry} if $f^{\dagger} \circ f$ is an idempotent.
    \end{itemize}
\end{defn}
\begin{lem}\label{lem:distributive}
    In a $W^*$-2-category composition of 1-morphisms distributes over finite direct sums; that is, we have unitary natural isomorphisms
    \begin{align*}
       (X_1 \oplus X_2) \otimes Y &\cong (X_1 \otimes Y) \oplus (X_2 \otimes Y)
       \\
       X \otimes (Y_1 \oplus Y_2) &\cong (X \otimes Y_1) \oplus (X \otimes Y_2)
    \end{align*}
\end{lem}
\begin{proof}
    We prove the first isomorphism; the second is shown similarly. Let $w_i: X_i \to X_1 \oplus X_2$ be the pair of isometric 2-morphisms defining the direct sum over the $X_i$'s. Now consider the pair of morphisms $w_i \otimes \id_Y: X_i \otimes Y \to (X_1 \oplus X_2) \otimes Y$. These are isometries by unitarity of composition, and obey the defining properties of a direct sum by unitarity and linearity of composition. 
\end{proof}
\subsection{\texorpdfstring{The 2-categories $\WAlg$ and $\WCat$}{The 2-categories W*-Alg and W*-Cat}}

We now give precise definitions of the 2-categories $\WAlg$ and $\WCat$ introduced in Section~\ref{sec:introw*}; 
$\WAlg$ first appeared in~\cite{Landsman2000}, under the name $[W^*]$, and was further studied in \cite{BDH14}. $\WCat$ was proven equivalent to $\WAlg$ in~\cite{Yamagami2007}; in that work these categories were referred to as $\mathrm{Fun}$ and $\mathrm{Bimod}$, respectively.

\subsubsection{\texorpdfstring{$\WCat$}{W*-Cat}}

We begin with the definition of $\WCat$, which is straightforward to state.
\begin{defn}
The $W^*$-2-category $\WCat$ is defined as follows:
\begin{itemize}
    \item Objects: $W^*$-categories.
    \item 1-morphisms: Unitary linear normal functors.
    \item 2-morphisms: Bounded natural transformations.
    \item Associator and unitors: Trivial. 
\end{itemize}
\end{defn}
\begin{notation}\label{not:functorcomp}
    In the 2-category $\WCat$ we write composition of functors using the tensor product symbol, going from left to right. For example, let $r,s,t$ be $W^*$-categories, and let $X: r \to s$ and $Y: s \to t$ be functors. Then the functor `$X$ followed by $Y$' is written as $X \otimes Y: r \to t$. Our 2-category $\WCat$ is therefore the opposite of the $W^*$-2-category defined using the common `$Y \circ X$' notation for composition of functors. 
\end{notation}

\subsubsection{Correspondences}\label{sec:correspondences}

We quickly recall the definition of self-dual normal $W^*$-correspondences over von Neumann algebras, which were referred to as \emph{rigged modules} in~\cite{Rie74}. These are equivalent to bimodules and will be useful in the proof of Theorem~\ref{thm:stinespring}, but their tensor product is somewhat intricate, involving two separate completions; therefore for the definition of $\WAlg$ we prefer bimodules and the Connes fusion tensor product. 

\begin{defn}[Hilbert modules]\label{def:hilbmodules}
Let $s$ be a von Neumann algebra. We call a complex vector space $X$ a \emph{ right semi-inner product $s$-module} if:
\begin{itemize}
    \item It has a right action of $s$, in the algebraic sense. 
    \item It possesses a map $\braket{-,-}: X \times X \to s$ which is:
        \begin{enumerate}
            \item $\mathbb{C}$-linear in its right index. 
            \item \label{num:semiinner2} For all $x,y \in X$, $\braket{x,y\cdot a} = \braket{x,y}a$.
            \item For all $x,y \in X$, $\braket{y,x} = \braket{x,y}^*$.
            \item For all $x \in X$, $\braket{x,x} \geq 0$.
        \end{enumerate}
\end{itemize}
If, additionally, $\braket{x,x} = 0 ~ \Rightarrow ~ x=0$, we call $X$ a \emph{right pre-Hilbert $s$-module}. 

If, furthermore, $X$ is complete in the norm 
\begin{equation}\label{eq:hilbmodulenorm}
||x|| := \sqrt{||\braket{x,x}||_s},
\end{equation}
we say that $X$ is a \emph{right Hilbert $s$-module}.

Left Hilbert $s$-modules may be defined similarly, except that the $s$-action is on the left. 
\end{defn}
\begin{defn}[Adjointable maps]
Let $X,Y$ be right Hilbert $s$-modules. We say that a map $f: X \to Y$ is \emph{adjointable} if there exists $f^{\dagger}: Y \to X$ such that:
$$
\braket{f^{\dagger}(y),x} = \braket{y,f(x)} ~~~~~ \forall~~x \in X,y \in Y
$$
An adjointable map is automatically $\mathbb{C}$-linear, $s$-linear, and bounded. The adjoint is unique, and taking the adjoint is involutive.

Equipped with the adjoint and the operator norm defined by~\eqref{eq:hilbmodulenorm}, the adjointable maps $X \to X$ form a unital $C^*$-algebra which we call $\L{}(X)$. 
\end{defn}
\begin{defn}[Correspondences]\label{def:correspondences}
We call a right Hilbert $s$-module $X$ equipped with a unital $*$-homomorphism $\pi_X: r \to \L{}(X)$, which is nondegenerate in the sense that $r \cdot X$ is norm-dense in $X$, an \emph{$r,s$-correspondence}. In order to simplify notation, for any $a \in r$ and $x \in X$ we write $a \cdot x:= \pi_X(a)(x)$.

We call an adjointable map between $r,s$-correspondences which intertwines the $r$-actions  an \emph{adjointable intertwiner}.

If, for every $x,y \in X$, the map $r \to s$ defined by $a \mapsto \braket{x,a \cdot y}$ is weak $*$-continuous, we say that $X$ is \emph{normal}.

If every norm-continuous map $X \to s$ is of the form $\braket{y,-}$ for some $y \in X$, we say that $X$ is \emph{self-dual}. Every self-dual normal correspondence $X$ has a predual and its algebra of adjointable operators $\L{}(X)$ is a von Neumann algebra.
\end{defn}
\noindent
The following construction will be useful in what follows. 
\begin{defn}[Tensor product of a right Hilbert module with a left module]\label{def:corresptensprod}
Let $r$ be a von Neumann algebra. Suppose that $X$ is a right Hilbert $r$-module and $Y$ is a left module for $r$; that is, a Hilbert space equipped with a normal $*$-homomorphism $r \to B(Y)$. One may then form a Hilbert space $X \otimes_r Y$ as follows:
\begin{itemize}
    \item Consider the algebraic tensor product $X \odot_r Y$. This is a complex vector space with a semi-inner product defined by linear extension of 
    \begin{align}\label{eq:innprodcorresptensordef}
    \braket{x \otimes y, x' \otimes y'}:= \braket{y, \braket{x,x'} \cdot y'}
    \end{align}
    \item Quotient by the subspace $\Sigma$ of norm-zero elements to obtain a pre-Hilbert space $X \odot_r Y/\Sigma$. 
    \item Complete $X \odot_r Y / \Sigma$ with respect to the norm to obtain a Hilbert space $X \otimes_r Y$.
\end{itemize}
If $X$ is an normal $r,s$-correspondence, $X \otimes_r Y$ inherits the structure of a normal left $r$-module.

Suppose $f: X \to X'$ is an adjointable map of right Hilbert $r$-modules and $g: Y \to Y'$ is a bounded intertwiner of left $r$-modules. We can then form a bounded linear map $f \otimes_r g: X \otimes_r Y \to X' \otimes_r Y'$, which is defined by $(f \otimes_r g)(x \otimes_r y):= f(x) \otimes_r g(y)$ on elementary tensors.
\end{defn}

\subsubsection{Bimodules}\label{sec:bimodules}

We now move onto the definition of the 2-category $\WAlg$, which is equivalent to $\WCat$, but with an algebraic rather than functorial presentation.
\begin{defn}
Let $r$ be a von Neumann algebra. As we have seen, we say that a pair of a Hilbert space $H$ and a normal (i.e. weak $*$-continuous) unital $*$-homomorphism $r \to B(H)$ is a \emph{left module} over $r$. We say that a pair of a Hilbert space $H$ and a normal unital $*$-homomorphism $r^{\mathrm{op}} \to B(H)$ is a \emph{right module} over $r$. We say that a module is \emph{faithful} if the $*$-homomorphism is injective. 
\end{defn}
\noindent
Next we introduce the standard form of a von Neumann algebra, which will in particular allow us to define identity 1-morphisms in the 2-category $\WAlg$.

\begin{defn}[{\cite{Haagerup1975}}]\label{def:standardform}
    The \emph{standard form} of a von Neumann algebra $r$ is a Hilbert space $L^2(r)$ equipped with faithful left and right module actions $\hat{\pi}_{L}: r \to B(L^2(r))$ and $\hat{\pi}_{R}: r^{\op} \to B(L^2(r))$, an antilinear involution $J: L^2(r) \to L^2(r)$ and a self-dual cone $P \subset L^2(r)$ such that:
    \begin{itemize}
    \item $J\hat{\pi}_L(r) J = \hat{\pi}_L(r)'$. (Recall that $'$ is the standard notation for the commutant of a subalgebra.)
    \item $J \hat{\pi}_L(c) J = \hat{\pi}_L(c^*)$ for all $c \in Z(r)$
    \item $J \xi = \xi$ for all $\xi \in P$
    \item $\hat{\pi}_L(a)J\hat{\pi}_L(a)JP \subseteq P$ for all $a \in r$
    \item $\hat{\pi}_R(a)\xi = J\hat{\pi}_L(a^*)J \xi$ for all $a \in r, \xi \in L^2(r)$
    \end{itemize}
    Here by \emph{self-dual cone} we mean that $P = \{\eta \in L^2(r)~|~ \braket{\eta,\xi} = 0~\forall \xi \in P\}$. The standard form of a von Neumann algebra always exists and is unique up to unique unitary isomorphism. 
\end{defn}
\begin{defn}[Bimodules over von Neumann algebras and intertwiners]
Let $r,s$ be von Neumann algebras. An $r,s$-bimodule is a Hilbert space $M$ equipped with commuting left and right module actions $\pi_{M,L}: r \to B(M)$ and $\pi_{M,R}: s^{\op} \to B(M)$, i.e. $\pi_{M,R}(s^{\op}) \subset \pi_{M,L}(r)'$. For conciseness we will write these module actions as $a \cdot m \cdot b := \pi_{M,L}(a)(\pi_{M,R}(b)(m))$ for $a \in r, b \in s, m \in M$.

For $A,B$-bimodules $M,N$, we call a bounded linear bimodule map $f: M \to N$ an \emph{intertwiner}. 
\end{defn}

\begin{notation}
From now on, in order to keep track of whether we are considering a Hilbert space as a left or right module, or both, we will use subscript lettering. For instance, ${}_rL^2(r)_r$, $L^2(r)_r$ and ${}_r L^2(r)$ are the Hilbert space $L^2(r)$ considered as a $r,r$-bimodule, a right $r$-module and a left $r$-module respectively. In this notation, $\Endo({}_rL^2(r)_r)$ is the space of bounded linear maps on $L^2(r)$ intertwining the left and right $r$-actions, whereas $\Endo(L^2(r)_r)$ is the space of such maps intertwining only the right $r$-action. 
\end{notation}
\noindent
We now recall the following proposition relating bimodules to the self-dual normal correspondences we defined in Section~\ref{sec:correspondences}.
\begin{prop}[{\cite[Thm. 2.2]{Baillet1988}}]\label{prop:baillet}
Let $r,s$ be von Neumann algebras. There is a unitary linear equivalence between the $W^*$-category of self-dual normal $r,s$-correspondences and adjointable intertwiners and the $W^*$-category of $r,s$-bimodules and intertwiners, specified as follows:
\begin{itemize}
    \item \emph{Correspondences to bimodules:} a self-dual normal $r,s$ correspondence $X$ is taken to the $r,s$-bimodule $X \otimes_s  L^2(s)$, where the tensor product is as in  Definition~\ref{def:corresptensprod}. Here the actions of $r$ and $s$ are specified by:
    $$a \cdot (x \otimes \xi) \cdot b := (a \cdot x) \otimes (\xi \cdot b)~~~~~~a \in r, b \in s$$
    An adjointable intertwiner $f: X \to Y$ is taken to the tensor product
    $$
    f \otimes_s \id_{L^2(s)} : X \otimes_s L^2(s)  \to Y \otimes_s L^2(s),
    $$
    defined as in Definition~\ref{def:corresptensprod}.
    \item \emph{Bimodules to correspondences:} an $r,s$-bimodule $X$ is taken to the self-dual normal $r,s$-correspondence $\Hom(L^2(s)_s, X)$, where the vector space $\Hom(L^2(s)_s, X)$ is considered as a right module over \\ $\Hom(L^2(s)_s, L^2(s)_s) \cong s$ by precomposition, and a left module over $r$ by postcomposition using the $*$-homomorphism $\pi_{L,X}: r \to B(X)$. 
    The $s$-valued inner product is specified by 
    $$\braket{f,g}:= \hat{\pi}_L^{-1}(f^{\dagger} \circ g).$$
    An intertwiner $f: X \to Y$ is taken to the adjointable intertwiner 
    $$(f \circ -): \Hom(L^2(s)_s, X) \to \Hom(L^2(s)_s,Y).$$
\end{itemize}
\end{prop}
\noindent
There are several equivalent ways to define Connes fusion of bimodules (see e.g.~\cite{Sauvageot1983,Wassermann1998,Takesaki2002,Sherm2003}). Here we use the definition given in~\cite{BDH14}.
\begin{defn}[Connes fusion of bimodules]\label{def:connesfusion}
Let $r$, $s$ and $t$ be von Neumann algebras, and let $M$ and $N$ be $r,s$- and $s,t$-bimodules respectively, with actions $\pi_{L,M}:r \to B(M)$, $\pi_{R,M}: s^{\op} \to B(M)$, $\pi_{L,N}: s \to B(N)$ and $\pi_{R,N}: t^{\op} \to B(N)$. Then the \emph{Connes fusion tensor product} $M \otimes N$ is defined as the completion of the algebraic tensor product 
$$
\Hom(L^2(s)_s,M) \odot_s {}_s L^2(s)_s \odot_s \Hom({}_s L^2(s),N)
$$
with respect to the norm defined by the semi-inner product
$$
\braket{f \otimes \xi \otimes g, f' \otimes \xi' \otimes g'} := \braket{\xi,  ((f^{\dagger} \circ f') 
 \circ (g^{\dagger} \circ g') ) (\xi')}
$$
following the quotient by the subspace of zero-norm vectors. Here $s$ acts on $\Hom(L^2(s)_s,M)$ by precomposition as in Proposition~\ref{prop:baillet}, and on \\ $\Hom({}_sL^2(s),N)$ by precomposition also, but this is a left action because $\Endo({}_sL^2(s)) \cong s^{\op}$. The left and right actions of $r$ and $t$ are specified as follows on elementary tensors:  
$$
a \cdot (f \otimes \xi \otimes g) \cdot b := (\pi_{M,L}(a) \circ f) \otimes \xi \otimes (\pi_{R,N}(b) \circ g). 
$$
One may equivalently define the Connes fusion tensor product more asymetrically as either of
\begin{align*}
\Hom(L^2(s)_s,M) \otimes_s N && M \tilde{\otimes}_{s} \Hom({}_{s} L(s),N)
\end{align*}
where the tensor product on the left is the tensor product of a right Hilbert $s$-module with a left $s$-module, as in Definition~\ref{def:corresptensprod}, and the tensor product on the right is the tensor product of a right $s$-module with a left Hilbert $s$-module, which may be defined analogously. 
\end{defn}
\begin{defn}[Connes fusion of intertwiners]
Let $T_M: M \to M'$ and $T_N: N \to N'$ be intertwiners of $r,s$- and $s,t$-bimodules respectively. Their Connes fusion tensor product is an intertwiner $T_M \otimes T_N: M \otimes N \to M' \otimes N'$ specified by the following map on elementary tensors:
$$
f \otimes \xi \otimes g \mapsto (T_M \circ f) \otimes \xi \otimes (T_N \circ g) 
$$
\end{defn}

\begin{defn}[Associator and unitors for Connes fusion]\label{def:connesassoc}
The associator for the Connes fusion tensor product is defined by the following sequence of unitary natural isomorphisms:
\begin{align*}
    &(M \otimes N) \otimes O \\
    &\cong (M \otimes N) ~\tilde{\otimes}_{t}~ \Hom({}_t L^2(t),O) \\
    &\cong (\Hom(L^2(s)_s,M) \otimes_s N) ~\tilde{\otimes}_{t}~ \Hom({}_t L^2(t),O)\\
    &\cong \Hom(L^2(s)_s,M) \otimes_s (N ~\tilde{\otimes}_{t}~ \Hom({}_t L^2(t),O)) \\
    &\cong \Hom(L^2(s)_s,M) \otimes_s (N \otimes O) \\
    &\cong M \otimes (N \otimes O)
\end{align*}
The tensor unit $r,r$-bimodule is ${}_rL^2(r)_r$, and the left unitor for the Connes fusion tensor product is defined by the following sequence of natural isomorphisms:
\begin{align*}
   L^2(r) \otimes N \cong \Endo(L^2(r)_r) \otimes_r N \cong  r \otimes_r N \cong N
\end{align*}
where the last isomorphism is given by extension of the usual identification of the algebraic tensor product $r \odot_r N \cong N$. The right unitor is defined similarly.
\end{defn}

\begin{prop}
Let $\WAlg$ be the $W^*$-2-category defined as follows:
\begin{itemize}
    \item Objects: Von Neumann algebras. 
    \item 1-morphisms: Bimodules. 
    \item 2-morphisms: Intertwiners of bimodules.
    \item Associator and unitors: As in Definition~\ref{def:connesassoc}.
\end{itemize}
Then $\WAlg$ is equivalent to $\WCat$.
\end{prop}
\begin{proof}
    The equivalence is shown in~\cite[Thm. 2.3]{Yamagami2007}. Note that Yamagami defines the equivalence as going from the opposite of the functorial 2-category, but because of our choice of notation for composition of functors (Notation~\ref{not:functorcomp}) we do not need to take the opposite 2-category. 
\end{proof}

\subsection{A classification of von Neumann algebras in \texorpdfstring{$\WAlg$}{W*-Alg}}

Equivalence classes of objects in $\WAlg$ correspond to Morita equivalence classes of von Neumann algebras. The up-to-isomorphism classification of von Neumann algebras is instead made at the level of endomorphism algebras of generating 1-morphisms.
\begin{defn}
    Let $\Hom(r_0,r)$ be the $W^*$-category of $1$-morphisms $r_0 \to r$. We say that a 1-morphism $X: r_0 \to r$ is a \emph{generator} in this category, or is \emph{generating}, if every other 1-morphism $r_0 \to r$ is a subobject of a (possibly infinite) direct sum of copies of $X$; that is, there exists an isometric 2-morphism from the object into the direct sum.
\end{defn}

\begin{defn}
Two von Neumann algebras $r$ and $s$ are \emph{Morita equivalent} if any of the following equivalent conditions hold:
\begin{enumerate}
    \item \label{num:morita1}There is an equivalence between the $W^*$-category of left $r$-modules and the $W^*$-category of left $s$-modules.
    \item \label{num:morita2}There is an equivalence between the $W^*$-category of right $r$-modules and the $W^*$-category of right $s$-modules.
    \item \label{num:morita3}$r$ and $s$ are equivalent as objects of $\WAlg$.
    \item \label{num:morita4}There exists some generating right $r$-module $M_r$ such that $s \cong \Endo(M_r)$.
\end{enumerate}
\end{defn}
\noindent
Here the implication $\eqref{num:morita1} \Leftrightarrow \eqref{num:morita2}$ is the content of~\cite[Prop. 8.9]{Rie74}, but follows immediately from the involutive structure at the 1-morphism level of the 2-category $\WAlg$ given by taking the conjugate of a bimodule. The implication $\eqref{num:morita2} \Leftrightarrow \eqref{num:morita3}$ follows from the equivalence 
of $\WAlg$ and $\WCat$.
The implication $\eqref{num:morita2} \Rightarrow \eqref{num:morita4}$ is as follows: there is a generating right $s$-module $L^2(s)_s$ such that $s \cong \Endo(L^2(s)_s)$, and the equivalence with right $r$-modules maps $L^2(s)_s$ to a generating right $r$-module satisfying the same condition. The implication $\eqref{num:morita4} \Rightarrow \eqref{num:morita2}$ can be seen by constructing an equivalence between right $r$-modules and right $s$-modules given that both categories have a generating object with isomorphic endomorphism algebras; we will show how to do this in the proof of Proposition~\ref{prop:vnclassification}.

Although we assume some background in category theory throughout this work, we recall the following definition since we will use it so often in what follows. 
\begin{defn}\label{def:equiv1-morphism}
We say that a 1-morphism $T: r \to s$ in a dagger 2-category is an \emph{equivalence}, or an \emph{equivalence 1-morphism}, if there exists a 1-morphism $T^{-1}: s \to r$ and unitary 2-isomorphisms $T \otimes T^{-1} \cong \id_{r}$ and $T^{-1} \otimes T \cong \id_s$.
\end{defn}
\noindent
Every equivalence $T: r \to s$ can be promoted to an \emph{adjoint equivalence}; that is, an equivalence 1-morphism with the same weak inverse $T^{-1}$, but such that the unitary 2-isomorphisms $T \otimes T^{-1} \cong \id_{r}$ and $T^{-1} \otimes T \cong \id_s$ obey the \emph{snake equations}:
\begin{align*}
\includegraphics[valign=c]{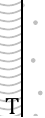}~~=~~\includegraphics[valign=c]{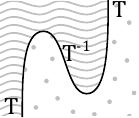}~~=~~\includegraphics[valign=c]{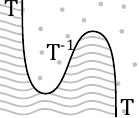}
\end{align*}
\begin{align*}
\includegraphics[valign=c]{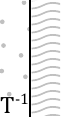}~~=~~\includegraphics[valign=c]{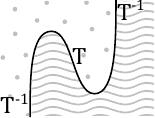}~~=~~\includegraphics[valign=c]{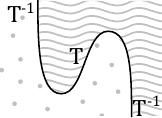}
\end{align*}
Here and throughout we have drawn the unitary 2-isomorphisms as cups and caps.
\begin{defn}
We say that two 1-morphisms $X: r_0 \to r$ and $Y: r_0 \to s$ in $\WAlg$ are \emph{unitarily equivalent} if there exists an equivalence 1-morphism $E: r \to s$ and a unitary 2-morphism $U: Y \to X \otimes E$. 
\end{defn}
\begin{prop*}[Restatement of Proposition~\ref{prop:vnclassification}]
Let $X: r_0 \to r$ and $Y: r_0 \to s$ be generating 1-morphisms in $\WAlg$. The von Neumann algebras $\Endo(X)$ and $\Endo(Y)$ are:
\begin{itemize}
    \item Morita equivalent iff the objects $r$ and $s$ are dagger equivalent.
    \item $*$-isomorphic iff there is a unitary equivalence $X \simeq Y$.
\end{itemize}
\end{prop*}
\begin{proof}
By the fourth characterisation of Morita equivalence above, the algebras $r$ and $\Endo(X)$ (resp. $s$ and $\Endo(Y)$) are Morita equivalent, since $X$ and $Y$ are generators in the category $\mathrm{Mod}$-$r$ of right $r$-modules and the category $\mathrm{Mod}$-$s$ of right $s$-modules respectively. Therefore there exist equivalences $E_X: r \to \Endo(X)$ and $E_Y: s \to \Endo(Y)$ in $\WAlg$ by the third characterisation of Morita equivalence, and the first bullet point follows. 

For the second bullet point, suppose that the algebras $\Endo(X)$ and $\Endo(Y)$ are isomorphic by some unital $*$-isomorphism $\phi: \Endo(X) \to \Endo(Y)$. Since $X$ is a generator of $\Mod$-$r$ and $Y$ is a generator of $\mathrm{Mod}$-$s$, we can define a unitary equivalence $\mathrm{Mod}$-$r \to \mathrm{Mod}$-$s$ which takes $X$ to $Y$, as follows. By definition of a generator, $\mathrm{Mod}$-$r$ is the idempotent completion of the subcategory whose objects are direct sums $\oplus_I X$, and likewise $\mathrm{Mod}$-$s$ is the idempotent completion of the subcategory whose objects are direct sums $\oplus_I Y$. We first define a functor $\hat{E}$ from one subcategory to the other:
\begin{itemize}
    \item On objects: $\hat{E}(\oplus_I X):= \oplus_I Y$.
    \item On morphisms: Consider a morphism $f: \oplus_I X \to \oplus_J X$. Let the isometries defining $\oplus_I X$, $\oplus_J X$, $\oplus_I Y$ and $\oplus_J Y$ be $\{v_i: X \to \oplus_I X\}_{i \in I}$, $\{v_j: X \to \oplus_J X\}_{j \in J}$, $\{w_i: Y \to \oplus_I Y\}_{i \in I}$ and $\{w_j: Y \to \oplus_J Y\}_{j \in J}$ respectively. Now observe that 
    $$
    f = \sum_{i \in I, j \in J} v_j \circ v_j^{\dagger} \circ f \circ v_i \circ v_i^{\dagger}
    $$
    and that the central terms $v_j^{\dagger} \circ f \circ v_i$ are in $\Endo(X)$. We then define 
    $$
    \hat{E}(f):= \sum_{i \in I, j \in J} w_j \circ \phi(v_j^{\dagger} \circ f \circ v_i) \circ w_i^{\dagger}.
    $$
\end{itemize}
That this is a unitary linear functor follows from the properties of a $*$-isomorphism and the defining equations of a direct sum. We observe that the functor between subcategories is full and faithful, since $\phi$ is an isomorphism, and essentially surjective. We now extend the functor between the subcategories to a functor $E: \mathrm{Mod}$-$r \to \mathrm{Mod}$-$s$. For any object $X$ of $\mathrm{Mod}$-$r$, let $v_X: X \to \oplus_{I} X$ be the isometry embedding it as a subobject of a direct sum over the generator. We define the functor $E$ as follows:
\begin{itemize}
    \item On objects: $X$ is taken to the object splitting the idempotent $\hat{E}(v_X \circ v_X^{\dagger}) \in \Endo(\oplus_I Y)$; that is, the object $E(X)$ (unique up to unitary equivalence) such that there is an isometry $w_X: E(X) \to \Endo(\oplus_I Y)$ satisfying $\hat{E}(v_X \circ v_X^{\dagger}) = w_X  \circ w_X^{\dagger}$.
    \item On morphisms: Let $f: X \to Y$ be a morphism. We have $f = v_Y^{\dagger} \circ v_Y \circ f \circ v_X^{\dagger} \circ v_X$. We then define $$E(f):= w_Y^{\dagger} \circ \hat{E}(v_Y \circ f \circ v_X^{\dagger}) \circ w_X.$$
\end{itemize}
This is straightforwardly seen to be a full and faithful unitary linear functor. To see that it is essentially surjective on objects, observe that every object in $\mathrm{Mod}-s$ is unitarily isomorphic to the splitting of some projection $\pi \in \Endo(\oplus_{i} Y)$; then consider the splitting of the projection $E^{-1}(\pi) \in \Endo(\oplus_{i} X)$ as the preimage. By the 
definition of $\WCat$,
we therefore have an equivalence $1$-morphism $E: r \to s$ in 
$\WCat$
such that $X \otimes E \cong Y$. By the equivalence of $\WCat$ and $\WAlg$, one direction is proven. 

In the other direction, suppose that $X$ and $Y$ are unitarily equivalent by a pair $(E, U: Y \to X \otimes E)$: then the Stinespring conjugation
$$
\includegraphics[valign=c]{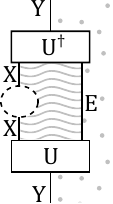}
$$
is clearly a $*$-isomorphism $\Endo(X) \to \Endo(Y)$ with the following inverse:
$$
\includegraphics[valign=c]{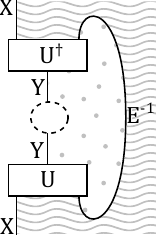}
$$
To show that this is indeed an inverse one can use the snake equations for the adjoint equivalence. 
\end{proof}

\section{Stinespring's theorem}

\begin{thm*}[Restatement of Theorem~\ref{thm:stinespring}]
Let $r,s$ be objects of $\WAlg$, and let $X: r_0 \to r$ and $Y: r_0 \to s$ be generating 1-morphisms. 

Let $f: \Endo(X) \to \Endo(Y)$ be a normal completely positive map. Then there exists a 1-morphism $E: r \to s$ (the `environment') and a 2-morphism $V: Y \to X \otimes E$ in $\WAlg$, such that 
\begin{align}\label{eq:stinespringdef}
f(a) = V^{\dagger} \circ (a \otimes \id_{E}) \circ V
\end{align}
We call the pair $(E,V)$ a \emph{representation} of the CP map $f$. The CP map $f$ is unital iff $V$ is an isometry.

Two pairs $(E,V)$ and $(E',V')$ are representations of the same CP map iff there exists a partial-isometric 2-morphism $\sigma: E \to E'$ such that 
\begin{align}\label{eq:stinespringintertwiner}
V' = (\id \otimes \sigma) \circ V && V = (\id \otimes \sigma^{\dagger}) \circ V'
\end{align}
Indeed, every CP map $f: \Endo(X) \to \Endo(Y)$ has a \emph{minimal representation}: that is, an initial object in the dagger category $\mathrm{Rep}(f)$ whose objects are Stinespring representations $(E,V: Y \to X \otimes E)$ of $f$ and whose morphisms $(E,V) \to (E',V')$ are 2-morphisms $\sigma: E \to E'$ in $\WAlg$ satisfying~\eqref{eq:stinespringintertwiner}.
\end{thm*}
\begin{proof}[Proof in a special case]
We begin by proving the result in the special case where the generating 1-morphisms under consideration are $X:= L^2(r)_r: r_0 \to r$ and $Y:= L^2(s)_s: r_0 \to s$. As we saw in Definition~\ref{def:standardform}, we have $\Endo(L^2(r)_r) \cong r$ and $\Endo(L^2(s)_s) \cong s$.

In this case the proof is similar to the usual proof of Stinespring's theorem for maps $r \to B(H)$ (see e.g.~\cite[Thm. 4.1]{Paulsen2003}), slightly modified since $s$ can here be an arbitrary von Neumann algebra, but in a way that would be obvious to an expert. Since this paper is targeted at an audience who may not have a strong background in operator algebra, we will provide full detail. We begin by defining the $r,s$-bimodule $E$. Consider the algebraic tensor product $r \odot s$. 
This vector space possesses an $s$-valued semi-inner product defined by linear extension of 
\begin{align}\label{eq:innerproddef}
\hpair{a_1 \ot b_1}{a_2 \ot b_2} := b_1^* f(a_1^* a_2) b_2~~~~~~~~a_1,a_2 \in r;~~b_1,b_2 \in s.
\end{align}
Linearity in the right index is clear. By a standard argument, positivity of $f$ implies that $f(a^*) = f(a)^*$ and therefore yields conjugate symmetry. Complete positivity of $f$ implies positive semidefiniteness as follows:
\begin{equation}
    \sum_{i,j} \hpair{a_i \ot b_i}{a_j \ot b_j} = \sum_{i,j}  b_i^* f(a_i^* a_j) b_j \ge 0.
\end{equation}
The space $r \odot s$ also possesses left and right actions of $r$ and $s$ respectively, given by linear extension of 
$$
a' \cdot (a \otimes b) \cdot b' := a'a \otimes bb' ~~~~~~~~a,a' \in r;~~b,b' \in s,
$$
such that the right $s$-action obeys the second compatibility relation w.r.t. the semi-inner product in Definition~\ref{def:hilbmodules}.

The semi-inner product defines a seminorm $||x||^2 := || \braket{x,x} ||_s $. Note that the left action of $r$ on $r \odot s$ is bounded with respect to this seminorm, since $f(a_1^*a^*aa_2) \leq ||a||^2f(a_1^*a_2)$ by positivity of $f$. We quotient $r \odot s$ by the subspace $\Sigma$ of elements with zero norm. The semi-inner product descends to an inner product on $r \odot s/ \Sigma$ by the Cauchy-Schwarz inequality
\begin{align}\label{eq:cschwarz}
\braket{x,y}\braket{y,x} \leq ||y||^2\braket{x,x}
\end{align}
which can straightforwardly be proven from the axioms of an $s$-valued semi-inner product. The left and right actions of $r$ and $s$ likewise descend to $r \odot s/\Sigma$ by boundedness of the left $r$-action and the inequality
\begin{align}\label{eq:boundedright}
||x \cdot b|| \leq ||x|| ||b||, 
\end{align}
which follows straightforwardly from the properties of an $s$-valued semi-inner product  listed in Definition~\ref{def:hilbmodules}. 

We then complete the quotient with respect to the norm induced by the inner product in a manner completely analogous to the completion of a pre-Hilbert space; the proofs are also completely analogous. That is, we define a new space $\hat{E}$ whose elements are equivalence classes of Cauchy sequences in $r \odot s/\Sigma$, where two Cauchy sequences $[x_n]:=x_1,x_2,\dots$ and $[y_n]:= y_1,y_2,\dots$ are said to be equivalent if $\lim_{n \to \infty} ||x_n - y_n|| = 0$. We extend the $s$-valued inner product to $\hat{E}$ by defining 
$$
\braket{[x_n],[y_n]} := \lim_{n \to \infty} \braket{x_n,y_n};
$$
this limit exists and yields a well-defined inner product by completeness of $s$ and the inequality
$$
||\braket{x,y}|| \leq ||x|| ||y||
$$
which follows from~\eqref{eq:cschwarz}. We define a right action of $s$ on $\hat{E}$ by 
$$
[x_n] \cdot b := [x_n \cdot b];
$$
this is well-defined by the inequality~\eqref{eq:boundedright}. We define a left action of $r$ on $\hat{E}$ by 
$$
a \cdot [x_n]:= [a \cdot x_n];
$$
this is well-defined by boundedness of the left $r$-action. It is straightforward to check that all axioms of a right Hilbert $s$-module are obeyed. We now claim that the left $r$-action is a unital $*$-homomorphism from $r$ into the $C^*$-algebra of adjointable intertwiners on $\hat{E}$. To see that the image of $r$ is adjointable, observe:
$$
\braket{[x_n],a\cdot [y_n]} = \lim_{n \to \infty}\braket{x_n,a \cdot y_n} = \lim_{n \to \infty}\braket{a^* \cdot x_n,y_n} = \braket{a^* \cdot [x_n],[y_n]},
$$
Here for the second inequality we have used the definition of the inner product~\eqref{eq:innerproddef}. This also shows immediately that the $r$-action is $*$-preserving. It is clearly a unital homomorphism. It is furthermore nondegenerate in the sense that $r \cdot \hat{E}$ is norm-dense in $\hat{E}$ (indeed, for this it is sufficient to consider the action on the elements $[1_r \otimes b]$). We have therefore obtained an $r,s$-correspondence (Definition~\ref{def:correspondences}). 

Finally, we claim that this $r,s$-correspondence is normal; that is, for each choice of $[x_n],[y_n] \in \hat{E}$ the map $\braket{[x_n],a\cdot [y_n]}: r \to s$ is weak $*$-continuous. For this, let $a_1 \otimes b_1, a_2 \otimes b_2$ be elementary tensors in $\hat{E}$; then 
$$
\braket{a_1 \otimes b_1, a \cdot (a_2 \otimes b_2)} = b_1^{*} f(a_1^{*}aa_2)b_2
$$
is weakly continuous in $a$, by normality of the CP map $f$ and weak continuity of left and right multiplication by a fixed element. Since the elementary tensors span a dense subset of $\hat{E}$, we claim that the map $\braket{[x_n],a\cdot [y_n]}$ is weakly continuous for any $[x_n],[y_n]$. For this, observe that the map $a \mapsto \braket{[x_n],a\cdot [y_n]}$ is weakly continuous iff for any net $a_{\alpha} \to a$ in $r$, we have $\braket{[x_n],a_{\alpha} \cdot [y_n]} \to \braket{[x_n],a \cdot [y_n]}$ in $s$. We will show that this holds for $\braket{x,a \cdot [y_n]}$, where $x$ is in the span of the elementary tensors; one can then extend to the case where $x$ is not in the span in a similar way. Let $p \in r_*$. Then:
\begin{align*}
&|p(\braket{x,a_{\alpha} \cdot [y_n]} - \braket{x,a \cdot [y_n]})| \\ &= 
|p(\braket{x,a_{\alpha} \cdot [y_n]} - \braket{x,a_{\alpha} \cdot y_n} + \braket{x,a_{\alpha} \cdot y_n} - \braket{x,a \cdot y_n}  + \braket{x,a \cdot y_n}  -\braket{x,a \cdot [y_n]})| \\
&\leq |p(\braket{x,a_{\alpha} \cdot [y_n]} - \braket{x,a_{\alpha} \cdot y_n})| + |p(\braket{x,a_{\alpha} \cdot y_n} - \braket{x,a \cdot y_n})| \\&~~~~~~~~~~~+ |p(\braket{x,a \cdot y_n}  - \braket{x,a \cdot [y_n]})|
\end{align*}
Observing that the first and last terms tend to zero as $n \to \infty$ by boundedness of $p$, and that the middle term tends to zero by weak continuity for elementary tensors, we obtain $p(\braket{x,a_{\alpha} \cdot [y_n]}) \to p(\braket{x,a \cdot [y_n]})$. 

We have therefore constructed a normal $r,s$-correspondence $\hat{E}$ (Definition~\ref{def:correspondences}). By~\cite[Prop. 6.10]{Rie74} this can be embedded as a weak $*$-dense subcorrespondence of a self-dual normal $r,s$-correspondence, which we call $\tilde{E}$. Finally, following the prescription of Proposition~\ref{prop:baillet}, we obtain a $r,s$-bimodule $E:= \tilde{E} \otimes_s L^2(s)$.

We must now define an intertwiner $V: L^2(s)_s \to L^2(r)_r \otimes E$ satisfying~\eqref{eq:stinespringdef}. 
Recall from the definition of Connes fusion (Definition~\ref{def:connesfusion}) that $L^2(r)_r \otimes E$ can be identified with the tensor product $\Endo(L^2(r)_r) \otimes_r E \cong r_r \otimes_r E$. We therefore define an intertwiner $V: L^2(s)_s \to r_r \otimes_r E$ as follows:
$$
V(\xi) := 1_r \ot 1_r \ot 1_s \ot \xi~~~~~~~\forall~ \xi \in L^2(s)   
$$
We show that $V^\dag(1_{r} \ot a \ot z \ot \xi) = (f(a)z) \cdot \xi$ for $z \in s$, as follows:
    \begin{align*}
        \hpair{&V^\dag(1_{r} \ot a \ot z \ot \xi)}{\xi'} = \hpair{1_{r} \ot a \ot z \ot \xi}{V(\xi')} \\&= \hpair{1_{r} \ot a \ot z \ot \xi}{1_{r} \ot 1_r \ot 1_s \ot \xi'} = \hpair{(f(a)z)\cdot \xi}{\xi'}
    \end{align*}
    Here for the last equality we have used the definitions~\eqref{eq:innprodcorresptensordef} and~\eqref{eq:innerproddef} of the respective inner products, and the fact that completely positive maps are $*$-preserving.
    Therefore, for $a \in r$, $\sigma \in L^2(s)$, we have that
    \begin{align*}
        (V^\dag & \circ  (\hat{\pi}_L(a) \ot \id_{E}) \circ V) (\sigma) = V^\dag (a \ot 1_r \ot 1_s \ot \sigma) 
        \\
        &= V^\dag (1_r \ot a \ot 1_s \ot \sigma)  = f(a) \cdot \sigma,
    \end{align*}
and so we obtain~\eqref{eq:stinespringdef}. This concludes the proof of existence of a Stinespring representation in the special case.

To finish the special case we must prove that the dagger category of Stinespring representations $\Rep{f}$ has an initial object. Again, this is constructed in much the same way as in the usual construction of a minimal representation (see e.g.~\cite[Prop. 4.2]{Paulsen2003}) for CP maps $r \to B(H)$. Let $(E,V: L^2(s)_s \to L^2(r)_r \otimes E)$ be a Stinespring representation. Since $L^2(r)_r \otimes E$ can be identified with the tensor product $\Endo(L^2(r)_r) \otimes_r E \cong r \otimes_r E \cong E_s$, we can consider $V$ as a map $V: L^2(s)_s \to E_s$. Define $\hat{E} \subseteq E$ to be the closure of the linear span of $\{a \cdot V(\xi)~|~a \in r, \xi \in L^2(s)\}$. Clearly $\hat{E}$ is a sub $r,s$-bimodule, and since the image of $V$ lies in $\hat{E}$, we obtain a new dilation $(\hat{E}, V)$ of $f$. For any other dilation $(E',V')$, a bimodule intertwiner $\sigma: \hat{E} \to E'$ satisfying~\eqref{eq:stinespringintertwiner} is defined and completely determined by
$$
\sigma(a \cdot V(\xi)) := a \cdot V'(\xi).
$$
Therefore the dilation $(\hat{E},V)$ is initial.

\emph{Proof in the general case.}
We now extend the result to any pair of generating 1-morphisms $X: r_0 \to r$ and $Y: r_0 \to s$, which is the only really new part of the proof. Let $A:= \Endo(X)$ and $B:= \Endo(Y)$. Since  $\Endo(X) = A \cong \Endo(L^2(A)_A)$ and $\Endo(Y) = B \cong \Endo(L^2(B)_B)$, by Proposition~\ref{prop:vnclassification} there exist equivalence 1-morphisms $E_X: r \to A$ and $E_Y: B \to s$ and unitary 2-morphisms $U_X : L^2(A)_A \ra X \ot {E_X}$ and $U(Y): Y \to L^2(B)_B \otimes E_Y$. We define
\begin{align*}
    &\mathrm{ad}_{U_X} = U_X^\dag \circ (- \otimes \id_{E_X}) \circ U_X : \Endo(X) \ra \Endo(L^2(A)_A), \nonumber \\
    &\mathrm{ad}_{U_Y} = U_Y^\dag \circ (- \otimes \id_{E_Y}) \circ U_Y : \Endo(L^2(B)_B) \ra \Endo(Y),
\end{align*}
both of which are $*$-isomorphisms, as we saw at the end of the proof of Proposition~\ref{prop:vnclassification}.
Now let $f': \Endo(L^2(A)_A) \to \Endo(L^2(B)_B)$ be the completely positive map such that the following diagram commutes:
    \begin{equation}\label{eq:commdiag}
\begin{tikzcd}
\Endo(X) \arrow[r,  "f"] \arrow[d, "\mathrm{ad}_{U_X}"] & \Endo(Y)   \\
\Endo ( L^2(A)_A ) \arrow[r, "f'"] & \Endo ( L^2(B)_B ) \arrow[u, "\mathrm{ad}_{U_Y}"]
\end{tikzcd}
\end{equation}
Let $\Rep{f}$ be the dagger category of Stinespring representations $(E,V: Y \to X \otimes E)$ of $f: \Endo(X) \to \Endo(Y)$, as defined in the statement of the Theorem; likewise, let $\Rep{f'}$ be the dagger category of Stinespring representations $(E',W: L^2(B)_B \to L^2(A)_A \otimes E')$ of $f': \Endo(L^2(A)_A) \to \Endo(L^2(B)_B)$. 

We now define functors $\Phi: \Rep{f} \to \Rep{f'}$ and $\Phi^{-1}: \Rep{f'} \to \Rep{f}$ as follows:
\begin{itemize}
    \item On objects:
\begin{align*}
\Phi\left(\includegraphics[valign=c]{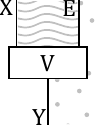}
\right)~~:=~~\includegraphics[valign=c]{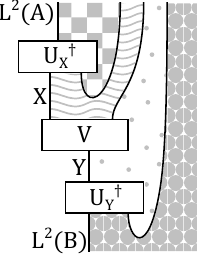}
\end{align*}
\begin{align*}
\Phi^{-1}\left(\includegraphics[valign=c]{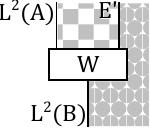}
\right)~~:=~~\includegraphics[valign=c]{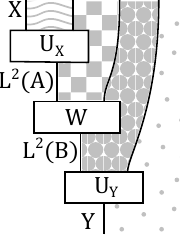}
\end{align*}
\item On morphisms:
\begin{align*}
\Phi\left(\includegraphics[valign=c]{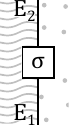}
\right)~~:=~~\includegraphics[valign=c]{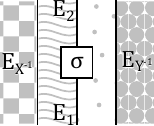}
\end{align*}
\begin{align*}
\Phi^{-1}\left(\includegraphics[valign=c]{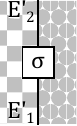}
\right)~~:=~~\includegraphics[valign=c]{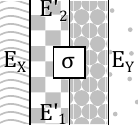}
\end{align*}
\end{itemize}
Here we have left the $r_0$-region blank and shaded the $r$-region with wavy lines, the $s$-region with polka dots, the $A$-region with checkerboard shading and the $B$-region with packed circles.
It is clear that $\Phi$ and $\Phi^{-1}$ are unitary functors. We claim that they are weak inverses witnessing an equivalence of $\Rep{f}$ and $\Rep{f'}$. To see this, observe:
\begin{align*}
(\Phi^{-1} \circ \Phi)\left(\includegraphics[valign=c]{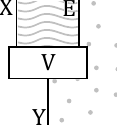}
\right)~~=~~\includegraphics[valign=c]{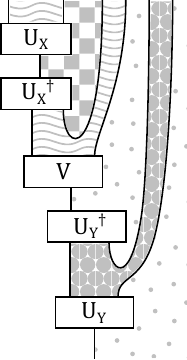}
~~=~~\includegraphics[valign=c]{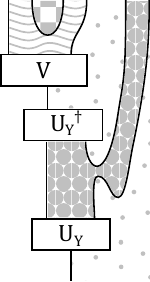}
\end{align*}
\begin{align}\label{eq:phiequiv}
~~=~~\includegraphics[valign=c]{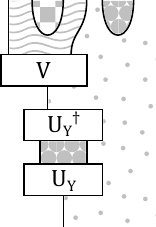}
~~=~~\includegraphics[valign=c]{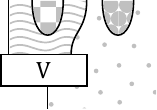}
\end{align}
Here for the first equality we used the definition of $\Phi$ and $\Phi^{-1}$; for the second equality we used unitarity of $U_X$ and a snake equation for the adjoint equivalence $E_Y$; for the third equality we used the fact that the cups and caps of the adjoint equivalence $E_Y$ are inverse; and for the fourth equality we used unitarity of $U_Y$. From the last expression of~\eqref{eq:phiequiv} we see that there is a unitary natural isomorphism $(\Phi^{-1} \circ \Phi) \cong \id_{\mathrm{Dil}(f)}$ whose component on $(E,V)$ is:
$$
\includegraphics[valign=c]{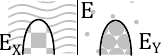}
$$
A similar argument holds for $(\Phi \circ \Phi^{-1})$. The categories $\Rep{f'}$ and $\Rep{f}$ are therefore equivalent. Since $\Rep{f'}$ is nonempty, then $\Rep{f}$ is also, hence we obtain existence of a Stinespring representation $(E,V: X \to Y \otimes E)$ for $f$; and since $\Rep{f'}$ has an initial object, then so does $\Rep{f}$. This completes the proof. 
\end{proof}
\noindent
As a first application of Stinespring's theorem we will consider $*$-homomorphisms in particular.
\begin{lem}\label{lem:unitarytominimal}
Let $(E,V: Y \to E \otimes X)$ be a Stinespring representation for a normal channel $f: \Endo(X) \to \Endo(Y)$. If the 2-morphism $V$ is unitary, then $(E,V)$ is a minimal representation.
\end{lem}
\begin{proof}
Let $A:= \Endo(X)$ and $B:=\Endo(Y)$, and  let $f'$ be the channel $\Endo(L^2(A)_A) \to \Endo(L^2(B)_B)$ making the diagram~\eqref{eq:commdiag} commute. The functor $\Phi: \Rep{f} \to \Rep{f'}$ maps the unitary $V: Y \to E \otimes X$ to a unitary $V': L^2(B)_B \to L^2(A)_A \otimes E'$ representing $f'$. Since an object of $\Rep{f}$ is mapped to an initial object by the equivalence $\Phi$ iff the object itself is initial, we need only show that $(E',V')$ is a minimal representation. Let $(\hat{E}',\hat{V}')$ be a minimal representation in $\Rep{f'}$ (we know that one exists) and let $\sigma: \hat{E}' \to E'$ be the unique isometry in $\Rep{f'}$. By~\eqref{eq:stinespringintertwiner} we have $V = (\id_{L^2(A)_A} \otimes (\sigma \circ \sigma^{\dagger})) \circ V' = V'$. But then by unitarity of $V'$, we have $\id_{L^2(A)_A} \otimes (\sigma \circ \sigma^{\dagger}) = \id_{L^2(A)_A} \otimes \id_{E'}$, which implies $\sigma \circ \sigma^{\dagger} = \id_{E'}$, since $L^2(A)_A \otimes E' \cong A \otimes_r E' \cong E'$. Therefore $\sigma$ is in fact unitary, and so $(E',V')$ and $(E,V)$ are minimal representations. 
\end{proof}

\begin{prop*}[Restatement of Proposition~\ref{prop:*hom}]
Let $f: \Endo(X) \to \Endo(Y)$ be a CP map and let $(E,V: Y \to X \otimes E)$ be its minimal Stinespring representation. The CP map is a unital $*$-homomorphism iff $V$ is unitary. 
\end{prop*}
\begin{proof}
We already know from Theorem~\ref{thm:stinespring} that the CP map is unital (i.e. a channel) iff $V$ is an isometry.
If $V$ is unitary, then the channel is clearly furthermore a $*$-homomorphism:
$$
\includegraphics[valign=c]{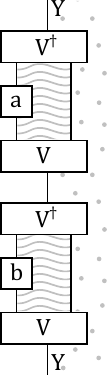}
~~=~~
\includegraphics[valign=c]{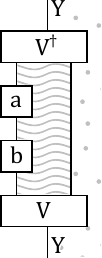}
$$
In the other direction, suppose that the channel is a unital $*$-homomorphism. Let $A:= \Endo(X)$ and $B:= \Endo(Y)$, and let $f': \Endo(L^2(A)_A) \to \Endo(L^2(B)_B)$ be the unital $*$-homomorphism making~\eqref{eq:commdiag} commute. We now observe that there is a canonical unitary representation of $f'$. Indeed, $f'$ defines an left $A$-module structure on $\Endo(L^2(B))$ compatible with the right $B$-module structure (c.f.~\cite[Rem. 3.5]{BDH14}), and the canonical representation is therefore $(_{A}L^2(B)_B,V)$, where $V$ is the canonical unitary isomorphism $L^2(A) \otimes {}_AL^2(B)_B \cong A \otimes_A  {}_AL^2(B)_B \cong L^2(B)_B$. By Lemma~\ref{lem:unitarytominimal}, this representation is minimal. Using the equivalence $\Phi^{-1}: \Rep{f'} \to \Rep{f}$ we obtain a unitary minimal representation of $f$, and the result follows. 
\end{proof}

\section{An application to extremal channels}

\begin{defn}
    Let $A,B$ be von Neumann algebras. We define $\mathrm{CP}(A, B; K)$ to be the convex set of normal CP maps $f: A \ra B$ such that $f(1_A) = K \in B$. In particular, $\mathrm{CP}(A,B ; 1_B)$ is the convex set of channels $A \to B$.
\end{defn}
\noindent
Recall that an element of a convex set is called \emph{extremal} if it cannot be written as a convex combination of two distinct elements of the set.

\begin{prop*}[Generalisation of Proposition~\ref{prop:choi}]
    Let $X:r_0 \to r$ and $Y: r_0 \to s$ be generating 1-morphisms in $\WAlg$. Let $f\in \mathrm{CP}(\Endo(X),\Endo(Y);K)$ be a completely positive map with minimal representation $(E,V: X \to Y \otimes E)$. Then $f$ is extremal iff, for any $m \in \Endo(E)$,
\begin{equation}\label{eq:extremalitycond}
\includegraphics[valign=c]{Pictures/Choi/extremalcond.pdf}~~=0 ~~~~~\Rightarrow~~~~~m=0
\end{equation}
\end{prop*}
\begin{proof}
  This proof follows a very similar approach to~\cite[Proof of Thm. 5]{Cho75} but does not rely on a decomposition of the channel in terms of Kraus operators. For `only if', suppose that $f$ is extremal and that $V^\dag \circ (\id_X \ot m) \circ V = 0$. We may assume without loss of generality that $m$ is Hermitian; indeed, we have that $V^\dag \circ (\id_X \ot (m \pm m^{\dagger})) \circ V = 0$, and if we prove that $m \pm m^{\dagger} = 0$ it will follow that $m=0$. By scaling $m \mapsto m/||m||$ we may further assume that $-\id_E \leq m \leq \id_E$.
  Define
    \begin{equation}
        f_{\pm}(a) = f(a) \pm V^\dag \circ (a \ot m) \circ V.
    \end{equation}    
   Clearly $f_{\pm}(1) = K$. Since $\id_E + m$ is positive, we can write $\id_E+m = \alpha^\dag \circ \alpha$ for some $\alpha \in \Endo(E)$. Then $f_{+}(a) = W^\dag \circ (a \ot \id_E) \circ W$, where $W = (\id_X \ot \alpha) \circ V$, and similarly for $\id_E - m$. Thus $f_{\pm}$ are completely positive and in $\mathrm{CP}(X,Y;K)$; so $f = \frac{1}{2} (f_{+} + f_{-})$, together with extremality of $f$, implies that $f = f_{+}=f_{-}$, and therefore that $V^{\dagger}\circ (a \otimes m)\circ V = 0$ for all $a \in \Endo(X)$.  Now let $\id_E-m = \beta^{\dagger}\circ \beta$, for some $\beta \in \Endo(E)$. Then $\beta$ is a morphism $(E,V) \to (E,V)$ in $\Rep{f}$. But by minimality of $(E,V)$ we see that $\beta$ is the identity on $E$, and so $m=0$. 
   
   For `if', suppose that~\eqref{eq:extremalitycond} holds, and that
    \begin{equation}
        f = \frac{1}{2} (f_1 + f_2) = \frac{1}{2} W^\dag \circ (a \ot \id_{E_1}) \circ W +\frac{1}{2} Z^\dag \circ (a \ot \id_{E_2}) \circ Z,
    \end{equation}
    for $f_1$, $f_2 \in \mathrm{CP}(X,Y;K)$, where $(E_1,W)$ and $(E_2,Z)$ are representations of $f_1,f_2$ respectively. By Lemma~\ref{lem:distributive} we obtain a representation $(E_1 \oplus E_2, \frac{1}{\sqrt{2}}(W \oplus Z))$ of $f$, where $W \oplus Z: Y \to X \otimes (E_1 \oplus E_2)$ is defined as $(\id_X \otimes \iota_1) \circ W + (\id_X \otimes \iota_2) \circ Z$; here $\iota_1: E_1 \to E_1 \oplus E_2$ and $\iota_2: E_2 \to E_1 \oplus E_2$ are the injections defining the direct sum. Since $(E,V)$ is a minimal representation of $f$, there is an isometry $(E,V) \to (E_1 \oplus E_2, \frac{1}{\sqrt{2}}(W \oplus Z))$ in $\Rep{f}$. Composing this isometry with the projection $\pi_1: E_1 \oplus E_2 \to E_1$, we obtain a morphism $\mu: E \to E_1$ such that $W = (\id_X \ot \mu) \circ V$. Now:
    \begin{equation}
        V^\dag \circ V = W^\dag \circ W = V^\dag \circ (\id_X \ot (\mu^\dag \circ \mu)) \circ V.
    \end{equation}
    It follows that $V^\dag \circ (\id_X \ot ((\mu^\dag\circ \mu) - \id_E))\circ  V = 0$, so $\mu$ is an isometry by~\eqref{eq:extremalitycond}. Therefore $f = f_1$. It follows that $f$ is extremal.
    \end{proof}

\bibliographystyle{alphaurl}
\bibliography{cpinf}

\end{document}